\documentclass[a4paper]{article}

\usepackage[backend=biber,style=alphabetic]{biblatex}
\usepackage{amsmath, amsfonts, amssymb, amsthm, mathtools, bbm} 
\usepackage{matlab-prettifier}
\usepackage{tikz, tikz-cd} 
\usepackage{tkz-euclide}
\usetikzlibrary{math,arrows.meta}
\usetikzlibrary{calc,intersections,positioning}
\usetikzlibrary {angles,quotes}
\newtheorem{theorem}{Theorem}
\newtheorem{proposition}[theorem]{Proposition} 
\newtheorem{lemma}[theorem]{Lemma}
\newtheorem{corollary}[theorem]{Corollary}
\theoremstyle{definition}
\newtheorem{definition}[theorem]{Definition}
\newtheorem*{defn*}{Definition}

\newtheorem{example}{Example}

\newcommand{\Z}{\ensuremath{\mathbb{Z}}} 
\newcommand{\F}{\ensuremath{\mathbb{F}}} 
\newcommand{\Ind}{\ensuremath{\text{Ind}}}

\newcommand{\Inf}{\ensuremath{\text{Inf}}}

\usepackage[english,greek]{babel}

\title{Non-Induced Representations in Characteristic Zero}
\author{Sean Dimmock}
\date{October 2026}

\begin{document}
\selectlanguage{english}
\maketitle

\begin{abstract}
This paper explores the failure of induction from all proper subgroups to generate the full representation ring, measured by a quotient group $Q_k(G)$. We extend existing results for cyclic groups to all finite abelian groups and calculate this quotient for all finite groups over fields of characteristic $0$ by reducing to the abelian case. We briefly discuss extensions to positive characteristics via modular representation theory. Finally, we detail our result for certain well-known families of groups.
\end{abstract}

\section{Introduction}
We study representations of a finite group $G$ via the \textit{representation ring} $R_k(G)$, where the elements are are integral combinations of isomorphism classes of $kG$-modules with addition and multiplication corresponding to the direct sum and tensor product. That is, for $[V],[U]\in R_k(G)$, we have;
\[
[V]+[U] = [V \oplus U] \ \text{and} \ [V]\cdot[U]=[V\otimes U].
\]

Formally, this is the split Grothendieck ring for the representations of the group $G$. This is different to the Grothendieck ring where we would have $[V]+[U]=[W]$ whenever
\[0\rightarrow V\rightarrow W\rightarrow U\rightarrow 0\]
is a short exact sequence. The split Grothendieck ring has a more restricted definition as it only accounts for exact sequences that split, and is thus better adapted for studying indecomposable representations. The elements of the representation ring are called \textit{virtual representations}, and the image of these under the character map to the class group $\mathcal{C}_G$ of G given by
\begin{align*}
    \chi:R_k(G)&\rightarrow\mathcal{C}_G\\\sum a_V[V]&\mapsto\sum a_V\chi_V
\end{align*}
are called \textit{virtual characters}.

We seek to understand $R_k(G)$ using the fundamental tool of induction. Induction from a subgroup $H<G$ gives rise to a well-defined additive homomorphism
\[\Ind_H^G:R_k(H)\rightarrow R_k(G).\]
Following from Corollary~$4.3.8(4)$, the image of this homomorphism is an ideal, and therefore it is natural to consider the quotients of $R_k(G)$ by these images for certain subgroups of $G$. We study objects of the form
\[\frac{R_k(G)}{\sum\limits_{H\in X}\Ind_H^GR_k(H)}\]
for $X$ some set of subgroups of $G$. A significant result in this area is Brauer's Induction Theorem, which asserts that the quotient is $0$ provided $X$ contains all elementary subgroups of $G$, with \textit{elementary} being defined as follows:
\begin{definition}
    A group $G$ is \textit{elementary} if it is the direct product of a $p$-group and a cyclic group with order coprime to $p$.
\end{definition}
For a thorough introduction to Brauer's Theorem and other relevant results see chapter~$10$ in \cite{Serre1977} or Section~$5.6$ in \cite{Benson1998vol1}. Moreover, J.A. Green showed in \cite{Green1955} in 1955 that no collection of subgroups smaller than the set of elementary subgroups can be used to make the quotient vanish. 

In particular, in this paper we consider the following construction:
\begin{definition}
    We define the \textit{ring of non-induced representations} to be the quotient ring
    \[Q_k(G):=\frac{R_k(G)}{I(G)}\]
    where
    \[I(G):=\sum_{H<G}I_H\]
    is known as the \textit{full induced ideal of $G$}, with
    \[I_H:=\Ind_H^GR_k(H)\]
    being the \textit{ideal induced by $H<G$}.
\end{definition}

It measures the failure of the proper subgroups of $G$ in inducing the virtual representations of $G$. Put another way, it describes the structure of the virtual representation classes of $G$ modulo induction. One use of this is to say $Q_k(G)=0$ if and only if every representation of $G$ is additively generated by induced representations. 

By Brauer's Induction Theorem and Green's result, $G$ is elementary if and only if $Q_k(G)\ne0$, since we are considering only proper subgroups, and so $G\notin X$. In Section~$2$, we will see this consequence of Brauer's Theorem in our results, but will not need to appeal to it in our proofs. In Section~$3$ we will assume this result.

To aid in understanding the representations of $G$, we want to understand the structure of $Q_k(G)$, hoping to relate it to the structure of $G$. In the following sections, we will work over characteristic $0$ unless otherwise stated. In Section~$2$, we will investigate the ring of non-induced representations for the finite abelian groups, extending the existing literature. In Section~$3$, we will extend this to all finite groups by reducing to the abelian case. Finally, in Section~$4$, we will provide some extra calculations of $Q_k(G)$ for certain families of groups that are particularly interesting, providing alternative family-specific proofs.  

I would like to thank Dr Robert Spencer for his advice and Professor Simon Wadsley for his help in verifying and improving this paper.

\section{Determining $Q_k(G)$ for Abelian Groups}

In this section we will determine $Q_k(G)$ for all finite abelian groups. 

\subsection{Cyclic Groups}

In 2021, Srihari considered this problem for the cyclic groups in the case of an algebraically closed field of characteristic $0$ in \cite{Srihari2021}, in which they proved:
\begin{theorem}[Theorem 1.1 in \cite{Srihari2021}]
    Let $C_n$ be a cyclic group of order $n$ and $k$ an algebraically closed field of characteristic $0$. Then
    \[Q_k(C_n)\cong\frac{\Z[X]}{(\Phi_n(X))}.\]
    where $\Phi_n(X)$ is the $n^{th}$ cyclotomic polynomial. In particular, the rank of the ring of non-induced representations is $\varphi(n)$.
\end{theorem}
    
Later that year, Jolliffe and Spencer extended this result using modular representation theory in \cite{JolliffeSpencer2021} to the case of an algebraically closed field of positive characteristic not dividing the order of the group. Much of the notation used in this paper is continued from \cite{JolliffeSpencer2021}.
\begin{corollary}[Corollary 3 in \cite{JolliffeSpencer2021}]
    Let $C_n$ be a cyclic group of order $n$ and $k$ an algebraically closed field of characteristic $p$ such that $p\nmid n$. Then
    \[Q_k(C_n)\cong\frac{\Z[X]}{(\Phi_n(X))}.\]
    In particular, \textup{$\text{rank}(Q_k(C_n))=\varphi(n)$}.
\end{corollary}

Firstly, that it agrees with the characteristic $0$ case is expected, but reassures us that we are dealing with a sensible object. Secondly, that it is independent of characteristic suggests we are seeing some information about the group. Therefore, we hope that in this and other cases, the ring of non-induced representations may be a helpful tool in understanding the group and its representations.

\subsection{Finite Abelian Groups}
We will extend the work of the previous section to finite abelian groups via the Structure Theorem.

A finite abelian group $G$ can be written as a product of the $p$-primary components over primes $p$. Explicitly, $G=\prod\limits_{p||G|}G_p$ where each $G_p$ is a $p$-group containing all elements of $G$ with order a power of $p$. If, for a fixed $G$, we have
\[G_p=C_{p^{a_1}}\times C_{p^{a_2}}\times\cdots\times C_{p^{a_{r_p}}},\]
we say $G$ has $p$-primary component given by $(a_1,\ldots ,a_{r_p})$. We say that $G$ has a \textit{non-cyclic $p$-primary component} if $r_p>1$, and we call $r_p$ the \textit{$p$-rank} of $G$. We partition the set of finite abelian groups into the following three categories:
\begin{enumerate}
    \item The group $G$ has no non-cyclic components ($r_p\le1$ for all primes $p$).
    \item The group $G$ has exactly one non-cyclic component ($r_p>1$ for only one prime $p$).
    \item The group $G$ has more than one non-cyclic component ($r_p>1$ for more than one prime $p$).
\end{enumerate}
By the Chinese Remainder Theorem, the first case is a cyclic group, so Section~$2.1$ applies. For the other cases, we can prove the following result:
\begin{lemma}
    Let $G$ be a finite abelian group with non-cyclic $p$-primary component and $k$ an algebraically closed field of characteristic $0$. Then $Q_k(G)$ is an $\F_p$-algebra.
\end{lemma}
\begin{proof}
    Let $A:=\hat{G}=\text{Hom}(G,k^\times)$ be the character group of the abelian group $G$. The integral group ring $\Z[A]$ has basis elements $e_\chi$ satisfying $e_\chi e_\psi=e_{\chi\psi}$. Therefore, there is a ring homomorphism defined as
    \begin{align*}
        \Phi:\Z[A]&\rightarrow R_k(G)\\
        e_\chi&\mapsto[\chi]
    \end{align*}
    where $\chi$ is the unique irreducible representation with character $\chi$, and extending linearly. This is naturally surjective because for abelian groups each representation is a direct sum of characters in $A$, and injective because the irreducible representations give a basis for $R_k(G)$, so there are no non-trivial relations of them.

    Since $G$ has non-cyclic $p$-primary component, we know that the $p$-torsion subgroup 
    \[A[p]=\{\lambda\in A:\lambda^p=1\}\]
    contains a subgroup $Y\le A[p]$ such that $Y\cong C_p\times C_p$. 
    
    For any subgroup $L\le Y$ of order $p$, let
    \[H_L:=L^\perp=\{g\in G:\lambda(g)=1\ \forall\lambda\in L\}.\]
    Since $|L|=p$, we have $[G:H_L]=p$, so $H_L<G$.

    Fix some irreducible character $\chi\in\hat{G}$. By Frobenius reciprocity, for each $\psi\in\hat{G}$ we have 
    \[
    \langle\Ind_{H_L}^G\chi|_{H_L},\psi\rangle_G=\langle\chi|_{H_L},\psi|_{H_L}\rangle_{H_L}
        =\begin{cases}1 & \psi|_{H_L}=\chi|_{H_L}\\ 0 & \psi|_{H_L}\neq\chi|_{H_L}\end{cases}
    .\]
    Therefore, the irreducible components of $\Ind_{H_L}^G\chi|_{H_L}$ are precisely the characters of $G$ extending $\chi|_{H_L}$, namely $\{\chi\lambda:\lambda\in L\}.$ Therefore, by considering dimension,
    \[\Ind_{H_L}^G(\chi|_{H_L})=\sum_{\lambda\in L}\chi\lambda.\]
    Hence, this sum is $0$ in $Q_k(G)$. There are $p+1$ subgroups of order $p$ in $Y\cong\F_p^2$, which we can view as lines in an $\F_p$-vector space. The identity element must be in all the lines, while every non-identity element of $Y$ must lie on precisely one line. Above, we demonstrated that each line-sum is $0$, so the sum over all lines is also $0$ and thus,
    \begin{align*}
        0 &= \sum_{L}\sum_{\lambda\in L}\chi\lambda\\
        &= (p+1)\chi+\sum_{\lambda\in Y\setminus\{1\}}\chi\lambda\\
        &=p\chi + \sum_{\lambda\in Y}\chi\lambda.
    \end{align*}
    However, the same Frobenius reciprocity argument gives
    \[\Ind_{H_Y}^G(\chi|_Y)=\sum_{\lambda\in Y}\chi\lambda\]
    which must vanish in $Q_k(G)$. Combining these results gives $p\chi=0$. This holds for any irreducible character $\chi$, and therefore 
    \[pQ_k(G)=0\]
    because the irreducible characters generate $Q_k(G)$. We know $Q_k(G)$ is a unital $\Z$-algebra, and so $pQ_k(G)=0$ implies that it is in fact an $\F_p$-algebra.
\end{proof}

This result has two immediate corollaries. Firstly, when $G$ has more than one non-cyclic $p$-primary component we must have $Q_k(G)=0$ since it cannot be an algebra over $\F_p$ for two distinct primes simultaneously. This is expected because $G$ is not elementary. Secondly, with exactly one non-cyclic $p$-primary component we see that $Q_k(G)\cong(\Z/p\Z)^N$ as an additive group for some $N$.

We next find the full ring presentation of $Q_k(G)$. The remaining non-trivial case is when $G$ has exactly one non-cyclic $p$-primary component. To reach this, we first want to consider abelian non-cyclic $p$-groups.
\begin{proposition}
    Let $G$ be a non-cyclic abelian $p$-group and $k$ an algebraically closed field of characteristic $0$. Then
    \[Q_k(G)\cong\F_p[A]/J(G)\] where $A:=\hat{G}$ is the character group of $G$, and the ideal $J(G)$ is the image of $I(G)$ in $\F_p[A]$ by reducing mod $p$.
\end{proposition}
\begin{proof}
     By Definition~$2$, to find $Q_k(G)$, we want to find the full induced ideal $I(G)$. Let $H<G$ and define the annihilator of $H$ to be 
    \[H^\perp=\{\chi\in A:\chi|_H=1\}.\]
    Let $\psi\in\hat{H}$ be a 1-dimensional representation of $H$, and $\tilde{\psi}\in A$ an extension of $\psi$, which exists by the Structure Theorem. By Frobenius reciprocity, we find
    \[\langle\Ind_H^G\psi,\chi\rangle_G=\langle\psi,\chi|_H\rangle_H=\begin{cases}1&\psi=\chi|_H\\0&\psi\neq\chi|_H\end{cases}\]
    and hence
    \begin{equation}
        \Ind_H^G\psi=\sum_{\lambda\in H^\perp}\tilde\psi\lambda.
    \end{equation}
    In the ring $\Z[A]$ we define the subgroup-sums 
    \[s_B:=\sum_{\lambda\in B}\lambda\]
    for each subgroup $B\le A$. Then we can rewrite equation~$(1)$ as
    \[\Ind_H^G\psi=\tilde\psi s_{H^\perp}.\]
    Therefore, we deduce that the full induced ideal can be generated by these subgroup-sums
    \[I(G)=(s_B:1\neq B\le A).\]
    So far, this holds for any abelian group since we are only using the fact that all representations are $1$-dimensional. We can improve the situation by also using the fact that $A$ is a $p$-group. Each non-trivial $B<A$ necessarily contains a subgroup $C\le B$ of order $p$, and therefore $B$ can be covered by cosets of $C$. The sum of the elements of a coset of $B\le A$ within $\Z[A]$ is a multiple of the subgroup-sum $s_B$. So, from our observation above, $s_C$ can be written as a linear combination of subgroup sums for subgroups of size $p$. Hence,
    \[I(G)=(s_C:C\le A,\ |C|=p).\]

    In Lemma~5 we showed that $pQ_k(G)=0$. Therefore, $pR_k(G)\subseteq I(G)$, and hence 
    \[Q_k(G)\cong\frac{\Z[A]/I(G)}{p(\Z[A]/I(G))}\cong\frac{\F_p[A]}{J(G)}\]
    where $J(G)$ is the image of $I(G)$ in $\F_p[A]$. The isomorphism is induced by the ring isomorphism in the proof of Lemma~5.
\end{proof}

Now the task is to make $J(G)$ more explicit. This reformulation allows us to make the following reduction.

\begin{lemma}
    Let $G$ be an abelian non-cyclic $p$-group with $p$-rank $r$ and $k$ an algebraically closed field of characteristic $0$. Then $J(G)=\F_p[A]\cdot J(B)$ where $A:=\hat{G}$ is the character group of $G$ and $B=A[p]\cong C_p^r$.
\end{lemma}
\begin{proof}
    Let $x_1,\ldots,x_r$ be generators for each cyclic factor of $A$, with $|x_i|=p^{a_i}$. Then $B$ is generated by $z_i:=x_i^{p^{a_i-1}}$ and each has order $p$. Every subgroup of order $p$ in $A$ must lie inside $B$. Therefore, every generator $s_C$ of $J(G)$ belongs to $\F_p[B]\subseteq\F_p[A]$ and the result follows.
\end{proof}
So we only need to consider the case of $G=C_p^r$. We will now solve this case explicitly.
\begin{proposition}
    Let $G=C_p^r$ for a prime $p$ and integer $r>1$. Then 
    \[Q_k(G)\cong\frac{\F_p[u_1,\ldots,u_r]}{(u_1,\ldots,u_r)^{p-1}}\]
    as $\F_p$-algebras where $u_i:=x_i-1$, and $x_i$ are the generators of each cyclic component of $A:=\hat{G}\cong C_p^r$ for $1\le i \le r$.
\end{proposition}
\begin{proof}
    We know from Lemma~5 that $R_k(G)\cong\Z[A]$ where $A:=\hat{G}$ is the character group. If $x_1,\ldots,x_r$ are generators of each cyclic component of $A$ then we can write
    \[\Z[A]\cong\frac{\Z[x_1,\ldots,x_r]}{(x_1^p-1,\ldots,x_r^p-1)}\] which reduces in characteristic $p$ to
    \[\F_p[A]\cong\frac{\F_p[u_1,\ldots,u_r]}{(u_1^p,\ldots,u_r^p)}\]
    where $u_i:=x_i-1$. Now, by Proposition~6, we wish to find $J(G)$. Let $\mathfrak{m}:=(u_1,\ldots,u_r)$. We will show that $J:=J(G)=\mathfrak{m}^{p-1}$.

    For a non-zero vector $a=(a_1,\ldots,a_r)\in\F_p^r$, define $g_a:=x_1^{a_1}\cdots x_r^{a_r}\in A$. Then $L_a:=\langle g_a\rangle$ is an order $p$ subgroup of $A$, and each subgroup of order $p$ can be written in this form. It has subgroup-sum given by
    \[s_{L_a}=1+g_a+\cdots+g_a^{p-1}\]
    which in characteristic $p$ becomes
    \[1+g_a+\cdots+g_a^{p-1}=(g_a-1)^{p-1}.\] These subgroup-sums generate $J$ since they cover all order $p$ subgroups of $A$, and thus
    \[J=((g_a-1)^{p-1}:a\neq0).\]
    But, by changing variables as before, we have
    \[g_a=\prod_{i=1}^r(1+u_i)^{a_i}\]
    which implies $g_a-1\in\mathfrak{m}$ and thus $(g_a-1)^{p-1}\in\mathfrak{m}^{p-1}$. Since we found that these forms generate $J$, this tells us that $J\subseteq\mathfrak{m}^{p-1}$.

    For the other direction, we begin by expanding this product:
    \[g_a-1=a_1u_1+\cdots+a_ru_r+(\text{higher order terms}).\]
    Define the linear form  $l_a:=a_1u_1+\cdots+a_ru_r$, and the remainder $h_a$ so that $g_a-1=l_a+h_a$ and $h_a\in\mathfrak{m}^2$. Therefore, 
    \[s_{L_a}=(g_a-1)^{p-1}=(l_a+h_a)^{p-1}.\]
    Since $h_a$ has degree at least $2$, the lowest degree term of the subgroup-sum $s_{L_a}$ is $l_a^{p-1}$. In other words, $s_{L_a}-l_a^{p-1}\in\mathfrak{m}^p.$ 
    
    Next, we will see that the $(p-1)$-powers of all linear forms span the set of all homogeneous degree-$(p-1)$ polynomials in $r$ variables. To see why, suppose we have a linear functional on the space of homogeneous polynomials of degree $p-1$ that vanishes on $l^{p-1}$ for every linear form $l$. This gives a homogeneous polynomial $F(a_1,\ldots,a_r)$ of degree $p-1$ which vanishes at all $a\in\F_p^r$. This can only be true if $F$ is identically zero. Therefore, the elements $l^{p-1}$ do indeed span. This means
    \[\frac{\mathfrak{m}^{p-1}}{\mathfrak{m}^p}\subseteq\frac{(J+\mathfrak{m}^p)}{\mathfrak{m}^p}.\]
    But the reverse inclusion is obvious from $J\subseteq\mathfrak{m}^{p-1}$, so
    \[J+\mathfrak{m}^p=\mathfrak{m}^{p-1}.\]
    Then the desired result follows by a simple application of Nakayama's Lemma.
\end{proof}

This is the crucial step in solving the general finite abelian group case, and we are not far from being able to write down a full statement. This particular result is also interesting as it reveals a rich structure.

First, we want to relate back to our original representation theory context. The ideal $\mathfrak{m}$ is naturally the kernel of the augmentation map
\begin{align*}
    \epsilon:\F_p[A]&\rightarrow\F_p\\
    \sum a_\chi e_\chi&\mapsto\sum a_\chi,
\end{align*}
and therefore we will call it the augmentation ideal of $\F_p[A]$.
The proof of proposition~8 shows that \(J(G)=\mathfrak m^{p-1}\). Hence
\[
Q_k(G)\cong \mathbb F_p[A]/\mathfrak m^{p-1}.
\]
Consider the \(\mathfrak m\)-adic filtration
\[
\mathbb F_p[A]\supseteq \mathfrak m\supseteq
\mathfrak m^2\supseteq\cdots.
\]
The associated graded ring is
\[
\operatorname{gr}_{\mathfrak m}(\mathbb F_p[A])
=\bigoplus_{d\geq 0}\mathfrak m^d/\mathfrak m^{d+1}.
\]
Since \(Q_k(G)\) is obtained by quotienting out \(\mathfrak m^{p-1}\), its dimension is the sum of the dimensions of the graded pieces. The associated graded ring is
\[
\operatorname{gr}_{\mathfrak m}(\mathbb F_p[A])
\cong \frac{\mathbb F_p[u_1,\ldots,u_r]}
{(u_1^p,\ldots,u_r^p)}.
\]
There are no restrictions for degree $d<p$, so the degree-\(d\) piece has dimension equal to the number of monomials of total degree \(d\) in \(r\) variables, namely \(\binom{r+d-1}{d}\). Therefore,
\[
\begin{aligned}
\dim_{\mathbb F_p} Q_k(G)
&=\sum_{d=0}^{p-2}\binom{r+d-1}{d}\\
&=\binom{r+p-2}{r}.
\end{aligned}
\]

\begin{example}
    Let $G=C_3^2$. Proposition~8 says
    \[Q_k(G)\cong\F_3[u,v]/(u,v)^2\]
    and has dimension $3$ as an $\F_p$-vector space. Therefore, if $u=[\chi]-[1]$ and $v=[\psi]-[1]$ then we have $u^2=uv=v^2=0$.
\end{example}

To finish dealing with finite abelian groups, we need to extend Proposition~8 to the general non-cyclic abelian $p$-group. We then want to understand how direct products behave in $Q_k(G\times H)$ to consider the cyclic factor in case $2$.

\begin{lemma}
    Let $G$ be a non-cyclic abelian $p$-group with $p$-primary component given by $(a_1,\ldots,a_r)$ and $k$ be an algebraically closed field of characteristic $0$. Then
    \[Q_k(G)\cong\frac{\F_p[u_1,\ldots,u_r]}{(u_1^{p^{a_1-1}},\ldots,u_r^{p^{a_r-1}})^{p-1}}\]
    as $\F_p$-algebras.
\end{lemma}
\begin{proof}
    By Proposition~6 we know $Q_k(G)\cong\F_p[A]/J(G)$ where $A:=\hat{G}$. Let $x_1,\ldots,x_r$ be generators for the cyclic factors of $A\cong G$. Then we change variables to $u_i:=x_i-1$ as before. We find
    \[\F_p[A]\cong\frac{\F_p[u_1,\ldots,u_r]}{(u_1^{p^{a_1}},\ldots,u_r^{p^{a_r}})}\] synonymously to the change of variables in Proposition~$8$.

    By Lemma~$7$ we have $J(G)=\F_p[A]\cdot J(B)$ where $B:=A[p]\cong C_p^r$. The group $B$ has generators $z_i=x_i^{p^{a_i-1}}$ for $1\le i\le r$, and each has order $p$. Using variables $v_i:=z_i-1$, Proposition~8 can be applied to get 
    \[J(B)=(v_1,\ldots,v_r)^{p-1}.\]
    Now, in characteristic $p$ we have
    \[v_i=x_i^{p^{a_i-1}}-1=(x_i-1)^{p^{a_i-1}}=u_i^{p^{a_i-1}}.\] Then the desired result follows by the result of Lemma~7 as stated above.
\end{proof}

We need one final result to be able to combine the results for each $p$-primary components in the second case.
\begin{lemma}
    Let $G$ and $H$ be groups with coprime order and $k$ an algebraically closed field of characteristic $0$. Then
    \[Q_k(G\times H)\cong Q_k(G)\otimes_{\Z} Q_k(H)\]
    as rings.
\end{lemma}
\begin{proof}
    Since $G$ and $H$ have coprime order, each subgroup of $G\times H$ can be split \[K=K_G\times K_H\] for subgroups $K_G\le G,\ K_H\le H$. Also, irreducible characters of $G\times H$ are precisely products of irreducible characters from each factor. Hence,
    \[R_k(G\times H)\cong R_k(G)\otimes_{\Z} R_k(H).\] Moreover, induction factors as \[\Ind_{K_G\times K_H}^{G\times H}(V\otimes W)\cong\Ind_{K_G}^GV\otimes\Ind_{K_H}^HW.\] Hence the full induced ideal is
    \[I(G\times H)=I(G)\otimes R(H)+R(G)\otimes I(H).\]
    Therefore
    \[\frac{R(G\times H)}{I(G\times H)}\cong\frac{R(G)}{I(G)}\otimes_{\Z}\frac{R(H)}{I(H)}\]
    and the desired result follows.
\end{proof}
We have proven everything needed to classify $Q_k(G)$ for all finite abelian groups.
\begin{theorem}
    Let $G$ be a finite abelian group and $k$ an algebraically closed field of characteristic $0$. Then
    \begin{enumerate}
        \item if $G$ is cyclic of order $n$ we have:
        \[Q_k(G)\cong\frac{\Z[X]}{(\Phi_n(X))}\]
        and has rank $\varphi(n)$.
        \item if $G$ has exactly one non-cyclic $p$-primary component, we can write $G=G_p\times C_n$ for $p\nmid n$ and $p$-primary component $G_p$ being given by $(a_1,\ldots,a_r)$, and we have:
        \[Q_k(G)\cong\frac{\F_p[X,u_1,\ldots,u_r]}{(\phi_n(X),(u_1^{p^{a_1-1}},\ldots,u_r^{p^{a_r-1}})^{p-1})}\]
        and has $\F_p$-dimension \[\dim_{\F_p}Q_k(G)=\varphi(n)\cdot p^{a_1+\ldots+a_r-r}{p+r-2\choose r}.\]
        \item if G has more than one non-cyclic $p$-primary component we have:
        \[Q_k(G)=0.\]
    \end{enumerate}
    
\end{theorem}

\section{Determining $Q_k(G)$ for a General Finite Group}

In this section, we will compute $Q_k(G)$ for all finite groups over characteristic $0$ by reducing to the abelian case. By Brauer's Induction Theorem, we know all non-elementary groups have a trivial ring of non-induced representations. Therefore, the only finite groups that were not covered in Section~$2$ and are not trivial are the elementary groups with non-abelian $p$-primary component. That is, $G\cong P\times C_n$ with $p\nmid n$ and $P$ a non-abelian $p$-group. To find $Q_k(G)$ in this case, we can apply Lemma~10 to decompose the product as long as we can find $Q_k(P)$ for any $p$-group $P$.

\begin{lemma}[Corollary 5.3.4 of \cite{Webb2016}]
    Let $P$ be a $p$-group and $k$ an algebraically closed field of characteristic $0$. Then each non-linear irreducible representation of $P$ is of the form $\Ind_H^P\lambda$ for some $H<P$ and $\lambda$ an irreducible linear character of $H$.
\end{lemma}
From which immediately follows:
\begin{corollary}
    The quotient $Q_k(P)$ is generated by classes of linear characters of $P$.
\end{corollary}
Thus the generators for $Q_k(P)$ are the same as for $Q_k(P/P')$ where $P'$ is the commutator subgroup of $P$. We can then ask whether $Q_k(P)$ can see any of the non-abelian structure of $P$. Put another way, do the non-linear representations introduce any other relations that don't come from the linear representations?

Consider the \textit{inflation map} on representations defined as
\[\Inf_{P/P'}^P(\rho):=\rho\circ\pi\]
where here we are thinking of a representation as a homomorphism
\[\rho:P/P'\rightarrow\text{GL}(V)\] and $\pi$ is the quotient map.
Inflation naturally preserves direct sums and tensor products, and hence induces a ring homomorphism
\[\Inf:R_k(P/P')\rightarrow R_k(P).\] Our question is then whether this map descends to the quotient ring of non-induced representations as an isomorphism. To break this down, we define a map \[\Psi:R_k(P/P')\rightarrow Q_k(P)\] as $\Inf$ composed with the quotient map $R_k(P)\rightarrow Q_k(P)$. Then Corollary~13 asserts that $\Psi$ is surjective. Therefore, to answer our question, we need to find the kernel of $\Psi$. 

\begin{lemma}
     Let $P$ be a $p$-group and $\Psi$ be defined as above. Then \[\ker(\Psi)=I(P/P')\] with $I$ being the full induced ideal as in Definition~2.
\end{lemma}
\begin{proof}
    Let $A:=P/P'$ be the abelianisation of $P$, so it is an abelian $p$-group, and let 
    \[\Psi:R_k(A)\rightarrow Q_k(P)\]
    be as above. We can identify $\hat{A}$ with the group of linear characters of $P$. Due to corollary~13, we want to consider the additive projection onto the subgroup generated by the linear characters of $P$\[l:R_k(P)\rightarrow R^{\mathrm{linear}}_k(P).\]

    Let $H<P$ and let $\lambda$ be an irreducible character of $H$. By Frobenius reciprocity, we have
    \[\langle\Ind_H^P\lambda,\chi\rangle_P=\langle\lambda,\chi|_H\rangle_H.\]
    Therefore, if $\lambda$ is not the restriction of a linear character of $P$, then this inner product is $0$ for all $\chi$. This implies
    \[l(\Ind_H^P\lambda)=0\]
    since $\Ind_H^P\lambda$ can have only non-linear irreducible components. So we assume $\lambda=\lambda'|_H$ for some linear character $\lambda'$ of $P$. Then, using Frobenius reciprocity, we deduce
    \begin{equation}l(\Ind_H^P\lambda)=\lambda's_{H^\perp}\end{equation}
    where 
    \[s_{H^\perp}=\sum_{\mu\in H^\perp}\mu\]
    is the subgroup-sum defined in the previous section. This agrees with our results in the abelian case, except in that case we would have no non-linear characters. 

    The linear characters of $P$ will factor through $A$, so we are interested in 
    \[\bar H:=\frac{HP'}{P'}\le A\] 
    because then $H^\perp=\bar H^\perp$. Therefore, by equation (2) every element of $l(I(P))$ is the inflation of an element of $I(A)$. Hence,
    \[l(I(P))\subseteq \Inf(I(A)).\]
    For the converse, let $K\le\hat{A}$ and consider $H=\pi^{-1}(K^\perp)$ where $\pi:P\rightarrow A$ is the quotient map. Then, by definition, $H^\perp=K$, so
    \[l(\Ind_H^P\mathbbm{1}_H)=s_K.\]
    The subgroup-sums generate $I(A)$, so 
    \[\Inf(I(A))\subseteq l(I(P))\]
    and hence
    \begin{equation}l(I(P))=\Inf(I(A)).\end{equation}

    Now, let $x\in\ker(\Psi)$. This means $\Inf(x)\in I(P)$. But $\Inf(x)$ is a combination of linear characters, so applying $l$ and equation (3) gives
    \[\Inf(x)\in\Inf(I(A)).\]
    But $\Inf$ is injective, so $x\in I(A)$ and hence
    \[\ker(\Psi)\subseteq I(A).\]
    The converse is clear, since every subgroup-sum in $I(A)$ must map to zero in $Q_k(P)$. Therefore,
    \[\ker(\Psi)=I(A).\]
\end{proof}
\begin{corollary}
    If $P$ is a $p$-group then $Q_k(P)\cong Q_k(P/P').$
\end{corollary}
\begin{proof}
    Apply the First Isomorphism Theorem to $\Psi$ and use Lemma~14.
\end{proof}
Therefore, we are only interested in the abelian group $P/P'$ and so we can apply Section $2$.
\begin{theorem}
    Let $G$ be a finite group and $k$ an algebraically closed field of characteristic $0$. If $G$ is not elementary then $Q_k(G)=0$. If $G$ is cyclic of order $n$ then
    \[Q_k(C_n)\cong\frac{\Z[X]}{(\Phi_n(X))}.\]
    Otherwise, we can write $G\cong P\times C_n$ for some prime $p$ and non-trivial $p$-group $P$, and some integer $n$ such that $p\nmid n$. Suppose $P/P'\cong C_{p^{a_1}}\times\cdots\times C_{p^{a_r}}$ for $r\ge1$ and $a_i>0$. Then
    \[Q_k(G)\cong\frac{\F_p[X,u_1,\ldots,u_r]}{(\phi_n(X),(u_1^{p^{a_1-1}},\ldots,u_r^{p^{a_r-1}})^{p-1})}.\]
\end{theorem}
\begin{proof}
    Corollary~15 reduces the problem to the abelian case. Then we can apply Theorem~11 and Lemma~10 to complete the calculation.
\end{proof}

\section{Extra Calculations of $Q_k(G)$}
In this section, we compute $Q_k(G)$ for two families of groups; symmetric and dihedral. We will find the result that follows from Theorem~16, and then find an elementary proof using our knowledge of the representations in each case.

\subsection{Symmetric Groups}
Here, we will consider the ring of non-induced representations for the symmetric groups $S_n$ for $n\ge 3$. An elementary group $G\cong P\times C_n$ will have $C_n\le Z(G)$ since $C_n\le G$ is normal and central. However, for $n\ge 3$ we have $Z(S_n)=\{e\}.$ It obviously can never be a $p$-group, so $S_n$ is not elementary for all $n\ge3$. Therefore, Theorem~16 in this case says:
\begin{theorem}
    Let $k$ be an algebraically closed field of characteristic $0$. Then 
    \[Q_k(S_n)=0.\]
    for every integer $n\ge 3$.
\end{theorem}
We are able to provide a direct proof of this using some theory for the representations of symmetric groups. To do this, we will use the lecture notes for the course \textit{Representation Theory of the Symmetric Groups} taught by Stacey Law \cite{law2022representation}.

\begin{proof}
    Let $S_\lambda=S_{\lambda_1}\times\cdots\times S_{\lambda_r}$ be the Young subgroup for a partition $\lambda\vdash n$. Let $S_n$ act on the cosets of this Young subgroup. If $S_\lambda<S_n$ (so $\lambda\ne(n)$)
    then its permutation representation is induced from the trivial representation:
    \[M^\lambda=\Ind_{S_\lambda}^{S_n}\mathbbm{1}\]
    and therefore it vanishes in $Q_k(S_n)$. The standard decomposition of Young permutation modules is:
    \[M^\lambda=\sum_{\mu\vdash n}K_{\mu\lambda}S^\mu\]
    where $S^\mu$ is the irreducible Specht module and $K_{\mu\lambda}$ is a Kostka number. Crucially, $K_{\lambda\lambda}=1$ and $K_{\mu\lambda}\ne 0\implies\mu\trianglerighteq\lambda$, so if we order by dominance then $(K_{\mu\lambda})$ is unitriangular.
    Therefore, since the permutation modules vanish, this gives 
    \[[S^\lambda]=0\]
    for every $\lambda\ne (n)$ by induction on dominance order.

    Now, it suffices to check the case $S^{(n)}=\mathbbm{1}$. Notice that for $n\ge3$, we have
    \[\Ind_{S_{n-1}}^{S_n}\mathbbm{1}=\mathbbm{1}\oplus S^{(n-1,1)}\]
    which translates in $Q_k(S_n)$ to
    \[[\mathbbm{1}]+[S^{(n-1,1)}]=0.\]
    But we have already shown the second term vanishes. Thus, $[\mathbbm{1}]=0$, and therefore
    \[Q_k(S_n)=0\]
    for $n\ge3$.
\end{proof}

This result is no surprise for the symmetric groups since it has many subgroups. More to the point:
\begin{quote}
    The problem with the symmetric groups is that there is too much symmetry and not enough group. 
    \\- G.D. James
\end{quote}
Therefore, it was not a stretch to find a simple proof for this case by exploiting the symmetry.

\subsection{Dihedral Groups}
Next, we will consider the ring of non-induced representations for the dihedral groups $D_{2n}$. It is well-known that
\[Z(D_{2n})=\begin{cases}\{e\}&\text{if $n$ odd}\\\{e,r^{n/2}\}&\text{if $n$ even}\end{cases}\]
Supposing $G\cong P\times C_n$, we know from Section $4.1$ that $C_n\le Z(G)$. So, if $n$ is odd then $Z(D_{2n})=\{e\}$ implies that, to be elementary, $D_{2n}$ must be a $p$-group, which is not possible. Therefore, $n$ must be even. Moreover, the center of a non-trivial $p$-group is non-trivial, which means $p\mid|Z(G)|$ by Lagrange's Theorem. Therefore, for $D_{2n}$ to be elementary we must have $n$ being a power of $2$ so that $D_{2n}$ is a $2$-group. So, in this case, Brauer's Induction Theorem implies
\[Q_k(D_{2n})=0\]
unless $n$ is a power of $2$. In this case, we have
\[D_{2n}/D_{2n}'\cong C_2\times C_2\]
and hence Corollary~15 gives
\[Q_k(D_{2n})\cong Q_k(C_2\times C_2)\cong \F_2\]
whenever $n$ is a power of $2$.

Therefore, our theorem has given us the full result for the dihedral groups. We will now provide an alternative elementary proof using the representation theory of the dihedral groups.
\begin{theorem}
    Let $D_{2n}$ be the dihedral group of order $2n$ and $k$ be an algebraically closed field of characteristic $0$. Then
        \[Q_k(D_{2n})\cong\begin{cases}\F_2&\text{if $n$ is a power of $2$}\\0&\text{otherwise.}\end{cases}\]
\end{theorem}
\begin{proof}
    The irreducible representations of $D_{2n}$ are:
    \begin{itemize}
        \item One-dimensional representations coming from $r\mapsto\pm1$ and $s\mapsto\pm1$ with $r^n=1$.
        \item Two-dimensional representations $V_j$ for $1\le j\le\lfloor(n-1)/2\rfloor$ given by
        \[r\mapsto\left(
\begin{matrix}
\zeta^j & 0 \\
0 & \zeta^{-j}
\end{matrix}\right), \ s\mapsto\left(\begin{matrix}
0 & 1 \\
1 & 0
\end{matrix}\right)\]
where $\zeta$ is a primitive $n$-th root of unity. Write characters as $\chi_j$.
\end{itemize}

A first important observation is the following. Consider the linear character $\lambda_j$ of $C_n$ defined by $\lambda_j(r)=\zeta^j$. Then, we have
\[\Ind_{C_n}^{D_{2n}}\lambda_j=V_j\] for each $j$. Therefore, the two-dimensional representations vanish in $Q_k(D_{2n})$, so we are only interested in the one-dimensional representations. 

If $n$ is odd then there are only $2$ one-dimensional representations. Namely, $\mathbbm{1}$ and $\epsilon$ where
\[\epsilon(r)=1,\ \epsilon(s)=-1.\]
In this case, we have
\begin{align*}
    \Ind_{\langle s\rangle}^{D_{2n}}\mathbbm{1}&=\mathbbm{1}\oplus \bigoplus_jV_j\\
    \Ind_{\langle s\rangle}^{D_{2n}}\epsilon&=\epsilon\oplus \bigoplus_jV_j.
\end{align*} 
which implies $[\mathbbm{1}]=[\epsilon]=0$ in $Q_k(D_{2n})$, and thus $Q_k(D_{2n})=0.$

If $n$ is even then there are $4$ one-dimensional representations: $\mathbbm{1},\epsilon,\delta,\epsilon\delta$ where
\[\delta(r)=-1,\ \delta(s)=1.\]
Now, we have
\begin{align*}
    \Ind_{\langle s\rangle}^{D_{2n}}\mathbbm{1}&=\mathbbm{1}\oplus\delta\oplus\bigoplus_jV_j\\
    \Ind_{\langle s\rangle}^{D_{2n}}\epsilon&=\epsilon\oplus\epsilon\delta\oplus\bigoplus_jV_j\\
    \Ind_{\langle rs\rangle}^{D_{2n}}\mathbbm{1}&=\mathbbm{1}\oplus\epsilon\delta\oplus\bigoplus_jV_j\\
    \Ind_{\langle rs\rangle}^{D_{2n}}\epsilon&=\epsilon\oplus\delta\oplus\bigoplus_jV_j
\end{align*}
which implies
\begin{align*}
    [\mathbbm{1}]+[\delta]&=0\\
    [\epsilon] + [\epsilon][\delta] &=0 \\
    [\mathbbm{1}] + [\epsilon][\delta] &=0 \\
    [\epsilon] + [\delta] &=0. \\
\end{align*}
By reducing these equations we find that all classes are equal in $Q_k(D_{2n})$, and $2[\mathbbm{1}]=0$. Therefore, $Q_k(D_{2n})=0$ or $\F_2$. To decide which is true, we need to determine when $[\mathbbm{1}]=0$.

Suppose $n$ is not a power of $2$. Write $n=2^am$ for $m>1$ odd, and consider the proper subgroup $H=\langle r^m,s\rangle\cong D_{2^{a+1}}.$ We will show that the only one-dimensional component of $\Ind_H^G\mathbbm{1}_H$ is $\mathbbm{1}_{D_{2n}}$. Notice, $\epsilon(s)=-1$ and $\delta(r^m)=(-1)^m=-1$, so none of $\epsilon,\delta,\epsilon\delta$ restrict to $\mathbbm{1}_H$ on $H$. Thus, by Frobenius reciprocity,
\[\langle\Ind_H^{D_{2n}}\mathbbm{1}_H,\chi\rangle=\begin{cases}1&\chi=\mathbbm{1}_{D_{2n}}\\0&\chi=\epsilon,\delta,\epsilon\delta\end{cases}\]
and therefore
\[\Ind_H^{D_{2n}}\mathbbm{1}_H=\mathbbm{1}_{D_{2n}}\oplus\bigoplus_jm_jV_j\] for some integers $m_j$. This implies $[\mathbbm{1}_{D_{2n}}]=0$

Finally, suppose $n$ is a power of $2$. To establish that $[\mathbbm{1}]\ne0$ we consider the following dimension-mod-$2$ map:
\begin{align*}
    \theta:R_k(D_{2^a})&\rightarrow\F_2\\ [V]&\mapsto\dim_k(V) \mod 2.
\end{align*}
It is a ring homomorphism by properties of dimension over direct sums and products. Let $H<D_{2^a}$ be a subgroup and $V$ a representation of $H$. Then
\[\dim_k\Ind_H^{D_{2^a}}V=[D_{2^a}:H]\dim_kV.\]
Since $D_{2^a}$ is a $2$-group, for every proper subgroup $H$ we have that $[D_{2^a}:H]$ is nonzero and even. Therefore, 
\[\theta([\Ind_H^{D_{2^a}}V])=0\] for all $H<D_{2^a}$, and thus
\[I(D_{2^a})\subseteq\ker\theta.\] This means that $\theta$ factors through the quotient and reduces to a homomorphism
\[\bar\theta:Q_k(D_{2^a})\rightarrow\F_2.\]
But we know $\bar\theta([\mathbbm{1}])=1$, so the map is nonzero, and hence the domain $Q_k(D_{2^a})$ is also nonzero. Therefore, we must have $Q_k(D_{2^a})\cong\F_2$.

\end{proof}

\printbibliography

\end{document}